\documentclass{amsart}
\usepackage{graphicx}
\usepackage{amssymb,amscd,amsthm,amsxtra}
\usepackage{latexsym}
\usepackage{epsfig}
\usepackage{mathtools}
\usepackage{esint}
\usepackage{color}

\newtheorem{thm}{Theorem}[section]

\newtheorem{lem}[thm]{Lemma}
\newtheorem{prop}[thm]{Proposition}
\theoremstyle{definition}
\newtheorem{defn}[thm]{Definition}
\theoremstyle{remark}
\newtheorem{rem}[thm]{Remark}
\numberwithin{equation}{section}
\newcommand{\be}{\begin{equation}}
\newcommand{\ee}{\end{equation}}

\newcommand{\R}{\mathbb R}

\newcommand{\eps}{\epsilon}

\newcommand{\p}{\partial}

\newcommand{\comment}[1]{}

\begin{document}

\title[Negative power potentials]{The Alt-Phillips functional for negative powers}
\author{D. De Silva}
\address{Department of Mathematics, Barnard College, Columbia University, New York, NY 10027}
\email{\tt  desilva@math.columbia.edu}
\author{O. Savin}
\address{Department of Mathematics, Columbia University, New York, NY 10027}\email{\tt  savin@math.columbia.edu}
\begin{abstract} We develop the free boundary regularity for nonnegative minimizers of the Alt-Phillips functional for negative power potentials
 $$\int_\Omega \left(\frac 1 2 |\nabla u|^2 + u^{\gamma} \chi_{\{u>0\}}\right) \, dx, \quad \quad \gamma \in (-2,0),$$
 and establish a $\Gamma$-convergence result of the rescaled energies to the perimeter functional as $\gamma \to -2$.
 \end{abstract}

\maketitle

\maketitle

\section{Introduction}

One fundamental problem in the calculus of variations consists in studying critical points for an energy functional of the type
$$ J(u, \Omega) = \int_\Omega\left( \frac 1 2 |\nabla u|^2 + W(u) \right)\, dx, $$
that is, solutions to its associated semilinear equation
$$ \triangle u=W'(u),$$
where $W: \mathbb{R} \to [0, \infty)$ represents a given potential with minimum $0$. Heuristically, minimizers of $J$ tend to concentrate their values near the zeros of $W$.

Certain classes of potentials have been extensively studied in the literature. One such example is the double-well potential $W(t)=(1-t^2)^2$ and the corresponding Allen-Cahn equation which appears in the theory of phase-transitions and minimal surfaces, see \cite{AlC,CH,MM}. 
When the potential $W$ is not of class $C^2$ near one of its minimum points, then a minimizer can develop constant patches where it can take that value, and this leads to a free boundary problem. Two such potentials were investigated in great detail. The first one is the Lipschitz potential $W(t)=t^+$ which corresponds to the classical obstacle problem, and we refer the reader to the book of Petrosyan, Shahgholian and Uraltseva \cite{PSU} for an introduction to this subject.
The second one is the discontinuous potential $W(t)=\chi_{\{t>0\}}$ with its associated Alt-Caffarelli energy (see \cite{AC}),  which is known as the Bernoulli free boundary problem or the two-phase problem. We refer to the book of Caffarelli and Salsa \cite{CS} for an account of the basic free boundary theory in this setting.    
 These two important examples are part of the more general family of Alt-Phillips potentials 
 $$ W(t)= (t^+)^\gamma, \quad \quad \gamma \in [0,2), \quad \quad \triangle u= \gamma u^{\gamma-1}.$$
Nonnegative minimizers $u \ge 0$ of $J$ for these power potentials, together with their free boundaries $$F(u):=\p\{u>0\},$$ were studied by Alt and Phillips in \cite{AP}. They showed that $F(u)$ has finite $n-1$ Hausdorff measure and established the regularity of the reduced part of the free boundary.  

In this paper we study the regularity properties of nonnegative minimizers of $J$ and their free boundaries for potentials of negative powers
$$ W(t)= t^{\gamma} \chi_{\{t>0\}} , \quad \gamma \in (-2,0).$$ 
 The bound $\gamma >-2$ is necessary for the existence of functions with bounded energy. To the best of our knowledge there are no available results in the literature addressing the more degenerate case of negative power potentials.
 
 The negative power potentials are natural in modeling sharper transitions of densities $u$ between their zero set and positivity set.  
With respect to the classical one-phase model, $\gamma=0$, in which the energy penalizes the measure of the set $\{u>0\}$ uniformly, in the negative power setting this penalization is stronger when $u$ is positive and small. This means that minimizers transition faster from the regions where $u \sim 1$ to their zero set and are not expected to be Lipschitz continuous near the free boundary. Despite this, we show that the free boundaries possess sufficiently nice regularity properties: they are $C^{1,\beta}$ surfaces up to a closed singular set of dimension $n-3$, see Theorem \ref{T3}.
We also analyze the Gamma convergence of the functionals $J$ when $\gamma$ tends to $ -2$ and establish the connection between these free boundary problems and the minimal surface equation, see Theorem \ref{T4}.

The change of sign in the exponent makes the problem more degenerate and the free boundary condition takes a different form. Some of the difficulties can be seen from a direct analysis of the one-dimensional case that we sketch below. 
Consider the ODE
$$ u''=\gamma u^{\gamma-1} \quad \quad \mbox{in } \quad (0,\delta) \subset \{u>0\}, \quad \quad u(0)=0.$$
The explicit homogenous solution
$$u_0= c_{\gamma} t ^ \alpha , \quad \alpha:=\frac{2}{2-\gamma},$$
for an appropriate constant $c_\gamma$, plays an important role in the analysis. 
In general, it follows that $u$  has an expansion of the form
\begin{align*}
 u(t)& = a \, t + o(t^{1+ \delta}), \quad \quad  \quad \quad \quad  \quad \mbox{if $\gamma \ge 0$,} \\
u(t)&= c_{\gamma} t ^ \alpha + a \, t ^{2-\alpha} + o(t^{2-\alpha+ \delta}), \quad \mbox{if $\gamma<0$,}
\end{align*}
with $a \in \R$ a free parameter.
 If in addition $u$ minimizes $J$ in $[-\delta,\delta]$, then $a=0$ in the case $\gamma >0$ and $a= \sqrt 2$ when $\gamma=0$. These can be viewed as Neumann conditions at the free boundary point $t=0$ and imply that the first nonzero term in the expansion of $u$ near $0$ is given precisely by the homogenous explicit solution $c_{\gamma} t ^ \alpha$, $\alpha \ge 1$.

On the other hand when $\gamma <0$, all solutions have $c_{\gamma} t ^ \alpha$ as the first nonzero term in the expansion. In this case the minimality condition imposes that the coefficient of the second order term in the expansion must vanish, i.e. $a=0$. Since $\alpha <1$, $u'$ becomes infinite at $0$, and this free boundary condition cannot be easily detected by integrations by parts or domain variations as in the case of nonnegative exponents. 

In this work we analyze minimizers by introducing an appropriate notion of viscosity solutions, and use the method of calibrations in order to handle the singularity of the PDE and of the free boundary condition. 

We remark that, after a change of variable $w=u^{\frac 1\alpha}$, the free boundary problem associated to the minimization of $J$ can be written as a degenerate one-phase problem in the form 
$$ \triangle w= \frac{h(\nabla w)}{w} \quad \mbox{in $\{w>0\}$,} \quad h(\nabla w)=0 \quad \mbox{on $F(w)$,}$$
where $h$ is an explicit quadratic polynomial vanishing on $\partial B_1$. The sign of $h$ in $B_1$ plays an important role in the stability of Lipschitz solutions near their free boundaries. The case $h>0$ in $B_1$ corresponds to positive power potentials and this problem was analyzed in greater generality in our previous work \cite{DS2}. When $h <0$ in $B_1$, it corresponds to negative power potentials which we study in this work. The free boundary condition needs to be understood in terms of the second term expansion near $F(w)$, see Section 7.

The paper is organized as follows. In the next section we state our main results. In Section 3 we discuss the existence and optimal regularity of minimizers. In Section 4 we introduce the notion of viscosity solutions and establish the non-degeneracy of minimizers. In Section 5 and 6 we perform a blow-up analysis based on the Weiss Monotonicity formula and on the $C^{1,\beta}$ regularity of flat free boundaries which is proved in Section 7.


\section{Main results}

In this section, we provide the statement of our main results. Since from now on we are only concerned with negative exponents, we change the notation from the Introduction and denote the negative exponent of the potential $W$ by $-\gamma$ with $\gamma \in (0,2)$. Precisely let
\begin{equation}
W(t):=\begin{cases} \frac{1}{\gamma} t^{-\gamma}\quad \text{if $t>0$},\\

\

0 \quad \text{if $t \le 0$,}
\end{cases}, \quad \quad \gamma \in (0,2),
\end{equation}
we consider the minimization problem for the energy functional
\begin{equation}\label{J}
J(u,\Omega):= \int_{\Omega} \left(\frac 1 2 |\nabla u|^2 + W(u)\right) \, dx,
\end{equation}
among all non-negative $u\geq0$ with a given boundary data $\phi \in H^1(\Omega).$ 
Whenever it does not create confusion, the dependence of $J$ on the domain $\Omega$ is dropped.

The corresponding Euler-Lagrange equation in the set $\{u>0\}$  is
\begin{equation}\label{eq1}
\triangle u = - u^{-(\gamma+1)}
\end{equation}
which has the explicit homogenous solution
$$ u_0 = c_0 (x_n^+)^\alpha, \quad \quad \alpha:= \frac{2}{\gamma+2}, \quad c_0:=[\alpha(1-\alpha)]^{-\frac {1}{\gamma+2}}.$$
We remark that the problem is invariant under the $\alpha$-homogenous scaling
$$ u_\lambda (x) = \lambda^{-\alpha} u(\lambda x),$$
i.e. $u_\lambda$ is a minimizer for $J$ in $\lambda^{-1} \Omega$.

We state our main results below. Positive constants depending only on $n, \gamma$ will be called universal. We start with existence and regularity of minimizers.

\begin{thm}[Existence and optimal regularity]\label{ER}
Let $\Omega$ be a Lipschitz domain. There exists a nonnegative minimizer $u$ of $J$ with boundary data $\varphi \in H^1(\Omega)$, $\varphi \ge 0$. Moreover, any minimizer is H\" older continuous of exponent $\alpha$, i.e. $u \in C^\alpha(\Omega)$.
\end{thm}

The second result concerns the non-degeneracy of minimizers around the free boundary $F(u):=\partial \{u>0\} \cap \Omega$.

\begin{thm}\label{T2}
Assume that $u\ge 0$ is a minimizer of $J$ in $B_1$ and $0 \in F(u)$. Then 
$$ c r^ \alpha \le \max_{\partial B_r} u \le C r^\alpha, \quad \quad \forall r \le 1/2,$$
with $c$, $C$ universal.
Moreover, the $\alpha$-homogenous rescalings $u_\lambda$ converge on subsequences $\lambda_n \to 0$ to a global $\alpha$-homogenous minimizer, i.e. a cone.
\end{thm}

In Section 6, Theorem \ref{T5},  we show that flat free boundaries are regular. Then the regularity of $F(u)$ depends on the classification of minimal cones in low dimensions. We establish that minimal cones are trivial in dimension $n=2$ and obtain the following partial regularity result.  

\begin{thm}\label{T3}
Let $u$ be a nonnegative minimizer for $J$ in $B_1$. Then 
 $$ \mathcal H^{n-1}(F(u) \cap B_{1/2}) \le C(n, \gamma) $$
 and $F(u)$ is locally $C^{1,\beta}$ except on a closed singular set $\Sigma_u \subset F(u)$ of Hausdorff dimension $n-3$. 
 \end{thm}

Finally we obtain a Gamma convergence result for appropriate multiples of the $J$ functional as $\gamma \to 2$. Let $\Omega$ be a bounded Lipschitz domain. We equip the space of nonnegative integrable functions 
$$X:=\{ u \in L^1(\Omega), u \ge 0\} $$ 
with the distance
$$d_X(u,v):=\|u-v\|_{L^1} + \|\chi_{\{u>0\}}-\chi_{\{v>0\}} \|_{L^1}.$$

\begin{thm}\label{T4}
As $\gamma \to 2^-$, the rescaled $J$ functionals
$$\mathcal J_\gamma(u):= c_\gamma \, J(u,\Omega),\quad \quad c_\gamma:=(1-\frac \gamma 2) \sqrt{\gamma/2} ,$$
Gamma converge in $X$ to the perimeter function
$$\mathcal P (u) = Per_\Omega (\{u>0\}).$$
Precisely,

a) if $u_n \to u$ in $X$ and $\gamma_n \to 2$, then $\liminf \mathcal J_{\gamma_n}(u_n) \ge \mathcal P (u)$;

b) given $u \in X$, there exists $u_n \to u$ in $X$ such that $\mathcal J_{\gamma_n}(u_n) \to \mathcal P (u)$.

\end{thm}


\section{Existence, Optimal Regularity, and Gamma Convergence}

This section contains the proofs of Theorem \ref{ER} and Theorem \ref{T4}.
The existence of a minimizer is achieved by standard methods in the calculus of variations, and we only sketch the proof. For simplicity we assume $\Omega= B_1$ and, given a boundary data $\phi\in H^1(B_1)$, $\phi \geq 0$, we set:
$$\mathcal A:= \{u \in H^1(B) \ : \ u\geq 0, \quad u=\phi \quad \text{on $\p B_1$}\}.$$

\begin{prop}\label{ex}There exists a minimizer $u \in \mathcal A$ to $J$ in $B_1$.
\end{prop}
\begin{proof} Let $\phi_h \geq 0$ be the harmonic replacement of 
$\phi$ in $B_1$ and set
$$w:= \phi_h + (1-|x|)^\alpha.$$ Recall that $\alpha=2/(\gamma+2).$
Then, since $$w^{-\gamma} \leq (1-|x|)^{-\gamma \alpha},$$ and $\gamma \in (0,2)$, we have that $J(w)<+\infty.$ Thus,
$0 \leq \inf_{\mathcal A} J < +\infty.$ Let $u_n \in \mathcal A$ be a minimizing sequence. Then, up to extracting a subsequence, 
$$u_n \rightarrow \bar u \quad \text{weakly in $H^1(B_1),$ strongly in $L^2(B_1)$, and almost everywhere in $B_1$.}$$ Moreover, it is immediate to show that at all points $x$ where $u_n(x) \to u(x)$, then 
$$W(\bar u) \leq \liminf_{n \to \infty}  W(u_n),$$ which together with Fatou's lemma gives that
$$J(\bar u) \leq \lim_n J(u_n),$$ and the proof is complete.
\end{proof}

Next we prove the optimal regularity result.

\begin{prop}\label{opt1} $u \in C^{\alpha}(B_1)$, with norm in $B_{1/2}$ bounded by a constant $C>0$ depending on $n,\|\phi\|_{H^1}.$

\end{prop}

In order to prove Proposition \ref{opt1}, we
denote by $$a(r):= r^{1-\alpha}\left(\fint_{B_r} |\nabla u|^2 dx\right)^{1/2}.$$ It is in fact  enough to obtain the next lemma, as the desired Proposition \ref{opt1} will then follow by standard Campanato estimates.

\begin{lem} There exist constants $M, \rho>0$ (depending on $n$) such that if $a(1) \geq M$ then $$a(\rho) \leq \frac 1 2  a(1).$$ In particular,
$$a(r) \leq C, \quad \forall r\leq 1,$$ for some constant $C$ depending on $\phi.$
\end{lem}
\begin{proof} Without loss of generality, after a multiplication by a constant, we can assume that $a(1)=1$ and $u$ minimizes 
$$\int_{B_1} \left(\frac 1 2 |\nabla u|^2 + \eps u^{-\gamma} \chi_{\{u>0\}}\right) \, dx ,$$ with $\eps$ small to be made precise later. Thus, using the competitor $w:= \phi_h + \eps (1-|x|)^\alpha$ as in the proof of Proposition \ref{ex}, we conclude that (for $C>0$ depending only on $n$ and changing from line to line),
$$\int_{B_1} |\nabla u|^2 dx \leq \int_{B_1} |\nabla w|^2 dx + C\eps,$$ with 
$$\int_{B_1} |\nabla \phi_h|^2 dx \leq 1= \int_{B_1} |\nabla u|^2 \,dx.$$This yields,
$$\int_{B_1} |\nabla u|^2 dx \leq \int_{B_1} |\nabla \phi_h|^2 dx + C\eps,$$ and using that $\phi_h$ is the harmonic replacement of $u$,
$$\int_{B_1} |\nabla (u- \phi_h)|^2 dx \leq C\eps.$$ Since $|\nabla \phi_h|^2$ is subharmonic, this implies that $$|\nabla \phi_h|^2 \leq C, \quad \text{in $B_{3/4}$.}$$ Therefore, the last two inequalities lead to 
$$\fint_{B_\rho} |\nabla u|^2 dx \leq C(\eps \rho^{-n} + 1), $$ and hence

$$a(\rho) \leq \frac 1 2,$$ as long as $\rho, \eps=\eps(\rho),$ are sufficiently small.
\end{proof}

We also state a simple energy bound for minimizers.

\begin{lem}\label{EB}
Assume $u$ is a minimizer for $J$ in $B_1$. Then
$$J(u,B_{1/2}) \le C(\|u\|_{L^\infty(B_1)}).$$
\end{lem}
\begin{proof}
Let $\varphi(x)=M [(|x|-\frac 12)^+]^\alpha$ with $M= C \|u\|_{L^\infty}$ so that $\varphi > \|u\|_{L^\infty}$ near $\p B_1$. Then
$$J(u,B_{1/2}) \le J(u, \{ \varphi \le u\} ) \le J(\varphi, \{ \varphi \le u\} )  \le J(\varphi, B_1 ) \le C(M).$$
\end{proof}

Next we discuss the one-dimensional case, which is the motivation for our definition of viscosity solution in the next section.

\subsection{The one-dimensional case} \label{the1d}

Assume that $u:[0, \delta] \to \R^+$ solves the ODE
$$ u''= - u^{-(\gamma +1)} \quad \mbox {in $(0,\delta)$, and $u(0)=0$}.$$
We multiply the equation by $u'$ and integrate, and deduce that in $(0,\delta)$
$$u'^2 - \frac 2 \gamma u^{-\gamma}= \mu,$$
for some constant $\mu \in \R.$ 
We rewrite the equation as,
$$ \frac{d}{dt} G(u(t)) =1,$$
where
$$G(s):=\int_0^s \left(\mu+ \frac 2 \gamma r^{-\gamma} \right )^{-\frac 12} dr.$$
Then,
$$G(s)= \sqrt{\frac{\gamma}{2}}\int_0^s(r^{\gamma/2}-\frac{\mu \gamma}{2}r^{3\gamma/2}) + O(r ^{5\gamma/2}) \quad dr ,$$ hence
$$G(s)=  \sqrt{\frac{\gamma}{2}} \left ( \alpha s^{\gamma/2 +1}-\frac{\mu\gamma}{3\gamma+2}s^{3\gamma/2 +1} \right)+ O(s ^{5\gamma/2+1}).$$
We can compute the inverse of $G$ near $0$ and, after a simple computation obtain
\begin{equation}\label{oned}
u(t)=G^{-1}(t)=c_0 t^\alpha +  \mu c_1 t^{2-\alpha} + O(t^\sigma), \quad \quad \sigma > 2-\alpha,
\end{equation}
with $\alpha = 2/(2+\gamma)$ and positive constants $c_0$, $c_1$ and $\sigma$ depending only on $\gamma$.

Assume further that the extension of $u$ by $0$ on the negative axis is a minimizer of $J$ in the interval $[-\delta,\delta]$. Then we show that $\mu=0$ and $u=u_0=c_0 (t^+)^\alpha$ is the explicit $\alpha$-homogenous solution. 

For this we compare $u$ with infinitesimal dilations with the same boundary data
$$ u_\lambda(t) := u(\delta + \lambda(t-\delta) ) $$
and $\lambda$ close to 1. Then
$$J(u_\lambda,[-\delta,\delta])= \int_{-\delta}^\delta \left(\frac \lambda 2 (u')^2 + \frac{1}{\lambda \gamma} u^{-\gamma}\right) dt$$ 
is minimal when $\lambda=1$ which means $(u')^2=(2/\gamma) u^{-\gamma}$, and that gives $\mu=0$. 

Alternatively, we could use Cauchy-Schwartz inequality and write
$$ J(u,[-\delta,\delta]) \ge (2 \gamma)^{-1/2} \int_{-\delta}^\delta u^{-\gamma/2} u' dt,$$
and the right hand side depends only on the values of $u$ at the end points. The equality occurs when $(u')^2=(2/\gamma) u^{-\gamma}$, i.e. $\mu=0$.

\subsection{The $\Gamma$ convergence as $\gamma \to 2$.} 

This last argument based on the Cauchy-Schwartz inequality can be used in higher dimensions to deduce the convergence result as $\gamma \to 2$. 

\begin{proof}[Proof of Theorem \ref{T4}]

We recall that the space of nonnegative integrable functions is denoted by
$$X=\{u: \Omega \to [0, \infty)| \quad u \in L^1(\Omega)\},$$ and is equipped with the distance
$$ d_X(u,v):=\|u-v\|_{L^1} + \|\chi_{\{u>0\}} -  \chi_{\{v>0\}}\|_{L^1}.$$
The proof is similar to the classical Modica-Mortola argument \cite{MM} for the Ginzburg-Landau functional. Notice that by Cauchy-Schwartz inequality
\begin{equation}\label{GC}
\mathcal J_\gamma(u) \ge (1-\frac \gamma 2)\int_\Omega u^{-\gamma/2 }|\nabla u| \, dx = \int_\Omega |\nabla u^{1-\gamma/2}| \, dx.
\end{equation}
If $u_n \to u$ in $X$ then $u_n^{1-\gamma_n/2} \to \chi_{\{u>0\}}$  in $L^1 (\Omega)$ and part a) follows from the lower semicontinuity of the $BV$-norm. 

For part b) let $\tilde u$ be a smooth function that approximates $u$ and let $E \subset \R^n$ be a set with smooth boundary which approximates $\{u>0\}$ in $\Omega$ (see \cite{M}), in the sense that $\|u-\tilde u\|_{L^1} \le \eps$, and
$$Per_\Omega (E) \le Per_\Omega(\{u>0\}) + \eps, \quad \|\chi_{E \cap \Omega} - \chi_{\{u>0\}}\|_{L^1} \le \eps, $$
$$\mathcal H^{n-1}(\partial E \cap \partial \Omega) =0, \quad \int_{\{u>0\} \setminus E} |u| \, dx \le \eps.$$
For $x \in \Omega$ we let $d(x)$ denote the distance to $\p E$ when $x \in E \cap \Omega$, and extend $d(x)=0$ when $x \in \Omega \setminus E$. 
We take $\delta$ small, such that $$Per_\Omega (\{d > s\}) \le Per_\Omega (E) + \eps \quad \mbox{ for} \quad  s \in [0,c_0 \delta^\alpha].$$ 

Define $$w(x):= c_0 \min \{ d(x), \delta\}^ \alpha, \quad v:= w + \varphi(d) \tilde u,$$
where $\varphi$ is a smooth function with $\varphi(s) =0$ if $s \le \delta$, $\varphi(s)=1$ if $s \ge 2 \delta$. 
Notice that $\{w>0\}=E$ and
$$\int_\Omega|v-u| dx \le \eps + |\Omega| \cdot \|w\|_{L^\infty} + \int_{\{0<d<2\delta\} }|u| + |\tilde u| dx,$$
which can be made arbitrarily small provided that first $\eps$ and then $\delta$ are chosen sufficiently small. 

We have $v=w$ in the set
$$D:=\{d< \delta\},$$
  while $|v|+|\nabla v| \le C (\delta, \tilde u)$ in the set $\Omega \setminus D$. If $\gamma_n \to 2$ then
$$ (2-\gamma_n) J (v, \Omega \setminus D) \to 0,$$
hence
 \begin{equation}\label{OO}
 \mathcal J_{\gamma_n}(v, \Omega)= \mathcal J_{\gamma_n}(w, D) + o(1).
 \end{equation}
The inequality \eqref{GC} is an equality for $w$ in the domain $D$:
$$\mathcal J_{\gamma_n}(w, D) = \int_{D} |\nabla w^{1-\gamma_n/2}|= \int_0^{a_\gamma} Per (\{w^{1-\gamma_n/2}>s\})ds   $$
with $a_\gamma=(c_0 \delta^\alpha)^{1-\gamma_n/2} \to 1$ as $n \to \infty$. In conclusion
$$\mathcal J_{\gamma_n}(w, D) \le Per _\Omega(E) + \eps + o(1),$$
which together with \eqref{OO} gives the desired statement.

\end{proof}

\section{Minimizers as Viscosity Solutions} 

In this section we consider the following degenerate one-phase ($u \geq 0$) free boundary problem:
\begin{equation}\label{v}\begin{cases}\Delta u = -u^{-(\gamma+1)} \quad \text{in $\{u>0\} \cap B_1,$}\\
u(x_0 + t\nu)= c_0 t^{\alpha} + o(t^{2-\alpha}) \quad \text{on $F(u):=\p \{u>0\} \cap B_1,$}
\end{cases}\end{equation}
with $t \geq 0$, $\nu$ the unit normal to $F(u)$ at $x_0$ pointing towards $\{u>0\}$, and 
\begin{equation}\label{a}
\alpha:= \frac{2}{\gamma+2}, \quad \quad c_0 := [\alpha(1-\alpha)]^{-\frac {1}{\gamma+2}}, \quad \gamma \in (0,2), \quad \alpha \in ( \frac 12,1).
\end{equation}

We start by introducing the notion of viscosity solution to \eqref{v}. As usual, we say that a continuous function $u$ touches a continuous function $\phi$ by above (resp. below) at a point $x_0$ if 
$$u \geq \phi \ \text{(resp. $u\leq \phi $)}\quad \text{in a neighborhood of $x_0$}, \quad u(x_0)=\phi(x_0).$$ Typically, if the inequality is strict (except at $x_0$), we say that $u$ touches $\phi$ strictly by above (resp. below). In our context, with $\phi \geq 0$, when we say that $u$ touches $\phi$ strictly by above at $x_0$, we mean that $u \geq \phi$ in a neighborhood $B$ of $x_0$ and $u>\phi$ (except at $x_0$) in $B \cap \overline{\{\phi>0\}}$ (and similarly by below we require the inequality to be strict in a neighborhood of $x_0$ intersected $\overline{\{u>0\}}$).

We now consider the class $\mathcal C^+$ of continuous functions $\phi$ vanishing on the boundary of a ball $B:=B_R(z_0)$ and positive in $B$, such that $\phi (x)= \phi( |x-z_0|)$ in $B$ and $\phi$ is extended to be zero outside $B$. We denote by $d(x):= dist(x, \p B)$ for $x$ in $B$ and $0$ otherwise. 
Similarly we can define the class $\mathcal C^-$, with $\phi$ being zero in the ball and positive outside, and $d(x):= dist(x, \p B)$ for $x \in B^c$ and $0$ otherwise.

\begin{defn}\label{defn} We say that a non-negative continuous function $u$ satisfies \eqref{v} in the viscosity sense, if 

1) in the set where $u>0$, $u$ is $C^\infty$ and satisfies the equation in a classical sense; 

2) if $x_0 \in F(u):= \p \{u>0\} \cap B_1$, then $u$ cannot touch $\psi \in \mathcal C^+$ (resp. $\mathcal C^-$) by above (resp. below) at $x_0$, with $$\psi(x):= c_0 d(x)^\alpha + \mu \, d(x)^{2-\alpha},$$ $\alpha, c_0$ as in \eqref{a} and $\mu>0$ (resp $\mu<0$).
\end{defn}

Next we show that a barrier as in the definition above can be modified so that it is a subsolution (supersolution) of the interior equation, and the touching is strict. 

\begin{lem}\label{strict} Let $u$ be a non-negative continuous function in $B_1$, such that $u$ touches $\psi \in \mathcal C^+$ by above at $x_0 \in F(u)$, and $$\psi:= c_0 d(x)^\alpha + \mu \,  d(x)^\beta,$$ with $\alpha, c_0$ as in \eqref{a}, and $\mu>0$. Then $u$ touches $\phi \in \mathcal C^+$ strictly by above at $x_0$, with 
\begin{equation}\label{p}\phi:= c_0 d^\alpha + \frac{\mu}{2} d^{2-\alpha} + d^\sigma, \quad 2-\alpha< \sigma <4 - 3 \alpha,
\end{equation} and for $d_0>0$ small, 
\begin{equation}\label{sub}\Delta \phi > -\phi^{-(\gamma+1)}, \quad \text{in the annulus $0< d(x) < d_0$}.\end{equation}
\end{lem}
\begin{proof} The first part of the claim is obvious after replacing $B$ with a ball of half its radius tangent at $x_0$, given that $\mu>0$ and $d$ is small. In order to prove \eqref{sub}, since $\phi$  is radially symmetric, we need to show that for $d>0$ sufficiently small, 
$$\phi'' - \frac{n-1}{R-d} \phi' > -\phi^{-(1+\gamma)}.$$
Indeed, we need,
$$c_0 \alpha(\alpha-1)d^{\alpha -2} + \frac{\mu}{2}(2-\alpha)(1-\alpha)d^{-\alpha}+  \sigma(\sigma-1)d^{\sigma-2} + O(d^{\alpha-1})> $$
$$ - (c_0 d^{\alpha})^{-(\gamma+1)} \left(1+ \frac{\mu}{2 c_0} d^{2 - 2\alpha} + \frac{1}{c_0} d^{\sigma-\alpha}\right)^{-(\gamma+1)}$$
By our choice of $c_0$ and we get,
$$1 - \frac {\mu}{ 2 c_0} \, \frac {2-\alpha}{\alpha} d^{2-2 \alpha} - \frac{C(\sigma)}{c_0}d^{\sigma -\alpha} + O(d) <$$
$$1- \frac{\mu}{2 c_0} (1+\gamma) d^{2 -2\alpha}-  \frac{1+\gamma}{c_0} d^{\sigma-\alpha} + O(d^{4(1 -\alpha)}) $$
with 
$$
 C(\sigma)=\frac{\sigma(1-\sigma)}{\alpha (1-\alpha)}.$$
The relation between $\alpha$ and $\gamma$ implies the the coefficients of the $d^{2-2\alpha}$ terms are equal, and $\sigma > 2- \alpha$ implies 
$C(\sigma) > 1+\gamma.$ Hence, the desired inequality holds for $d$ small enough, since the upper bound on $\sigma$ gives $\sigma-\alpha < 4(1-\alpha) < 1$.
\end{proof}

We can now prove the following optimal regularity statement.

\begin{lem}\label{opt} Let $u$ be a viscosity solution to \eqref{v} in $B_1$ and assume that $F(u) \cap B_{1/2} \neq \emptyset.$ Then,
\begin{equation}\label{univ}u(x) \leq C dist(x, F(u))^\alpha, \quad \text{in $\{u>0\} \cap B_{1/2},$}\end{equation} for $C=C(n,\alpha)>0.$
\end{lem}
\begin{proof} 
Let $x_0 \in \{u>0\} \cap B_{1/2}$ and let $r:= dist(x_0, F(u))$. Consider the rescaling
$$\tilde u (x): = \frac{u(x_0 + rx)}{r^\alpha},$$
which solves \eqref{v} in $B_{1}$, and let us show that
$$\tilde u(0) \leq M$$
with $M>0$ universal and $B_1 \cap F(\tilde u) = \{\bar x\}$. Assume by contradiction that $\tilde u(0) > M$, with $M$ to be made precise later. Notice that,
$$\Delta \tilde u \leq 0, \quad \Delta (\tilde u-1)^+ \geq -1 \quad \text{in $B_1$}.$$ Thus, by the mean value inequality for subharmonic functions, we get that 
$$\int_{B_{1/2}}(\tilde u -1)^+ \;dx \geq c_1 M.$$
with $c_1$ depending only on $n$. Then
$$\int_{B_{1/2}}\tilde u \;dx\geq c_1 M,$$ and, using that $\tilde u \ge 0$ is superharmonic we have that $\tilde u \ge c_2 M$ in $B_{1/2}$. After iterating this result a few times we find 
$$\tilde u \geq  c_3 M \quad \text{in $B_{1-d_0}$},$$ with $c_3$ depending on $d_0$ and with $d_0$ universal to be specified below.
Now, let $\phi$ be as in equation \eqref{p}, with $\mu=1$. The computation in Lemma \ref{strict} proves that $$\Delta \phi > - \phi^{-(\gamma+1)} \quad \text{in $B_1\setminus \overline{B}_{1-d_0}$},$$ for $d_0$ universal. Moreover, if $M>0$ is large enough universal,
$$u \geq c_3 M \geq \phi \quad \text{on $\p B_1 \cup \p B_{1-d_0}.$}$$ and the maximum principle implies $\tilde u \geq \phi$ in $B_1 \setminus B_{1-d_0}$. On the other hand, the two functions touch at $\bar x$, which contradicts the definition of viscosity solution.
\end{proof}

Next we show that minimizers to \eqref{J} are indeed viscosity solutions.

\begin{prop}\label{min=vis} Let $u$ minimize \eqref{J} in $B_1$. Then $u$ is a viscosity solution to \eqref{v}.
\end{prop}
\begin{proof}
The fact that $u$ is continuous and satisfies the first equation in \eqref{v} in $B_1 \cap \{u>0\}$ follows from Proposition \ref{opt1}. It remains to show that $u$ satisfies the free boundary condition. 

Let us assume that $u$ touches $\psi$ by above at $x_0 \in F(u)$, with $\psi$ as in Definition \ref{defn} and $\mu >0$. Then in view of Lemma \ref{strict}, $u$ touches $\phi$ strictly by above at $x_0$, with $\phi$ defined in \eqref{p}. We will show that this contradicts the minimality of $u$, using a calibration argument. For simplicity, assume that the unit normal to $F(u)$ at $x_0$ is $e_n$.
For any non-negative function $v$, smooth in its positivity set, we denote by $\Gamma_v$ its graph in $\R^{n+1}$ over the positivity set, and by $\nu_v(x)$ the upward unit normal to $\Gamma_v$ at $(x,v(x))$. 

Notice that we can write the energy of $u$ over a domain $\Omega$ as a surface integral over its positivity graph in $\Omega$, $\Gamma_u(\Omega)$, in the following way:
\begin{equation}\label{EE}J(u, \Omega) = \int_{\Gamma_u(\Omega)} G(u, \nu_u) d\sigma,\end{equation}
with 
$$G(s, \nu):= \frac{1}{2} \frac{|\nu'|^2}{\nu_{n+1}} + W(s) \nu_{n+1}, $$
and
$$s>0,\quad  |\nu|=1,  \quad \nu:=\langle \nu', \nu_{n+1}\rangle, \quad \nu_{n+1}>0.$$

Let $G(s,y)$ is the 1-homogeneous extension (in $y$) of $G(s,\nu)$.
Then, 
$$\nabla_y G(s,\nu):= \langle \frac{\nu'}{\nu_{n+1}}, -\frac 12 \frac{|\nu'|^2}{\nu_{n+1}^2 } + W(s) \rangle,$$
and the homogeneity and convexity in $y$ imply,
\begin{equation}\label{hom}G(\phi(x), \nu_\phi(x)) = V_\phi(x,\phi(x)) \cdot \nu_{\phi}(x), 
\end{equation}
with
$$V_\phi(x,\phi(x)):= \nabla_y G(\phi(x), \nu_\phi(x)),$$
and
\begin{equation}\label{con}G(\phi(x), \nu_u(x)) \geq  V_\phi(x,\phi(x)) \cdot \nu_{u}(x).\end{equation}

The vector field $V_\phi(x,\phi(x))$ is defined on the graph $\Gamma_\phi$, and we extended in $\R^{n+1}$ constantly in the $e_n$ direction and denote it simply by $V$. This vector field is associated with the graphs of the translations 
$$\phi_t(x):= \phi(x+te_n), \quad t \in \R,$$ which provide a foliation of a neighborhood of $(x_0, 0)$ in $\R^n \times \R^+$. In other words,  for each given point $X:=(x,x_{n+1})$ in this set, we identify the element $\phi_{t_X}$ of the foliation that passes through it i.e $x_{n+1}= \phi_{t_X}(x)$ and $$V(X) = V_{\phi_{t_X}}(X).$$
Now, set $$D:= \{u(x) < x_{n+1} < \phi_{\bar t}(x)\} \quad \subset \R^{n+1},$$  with $\bar t>0$ chosen in such a way that $D$ is included in the neighborhood of $(x_0,0)$ foliated by the graphs of the $\phi_t$'s.
Denote by $$D_\eps:= D \cap \{x_{n+1} > \eps\}, \quad \mbox{ and} \quad \Gamma_\eps:=D \cap \{x_{n+1}=\eps\},$$ for $\eps>0$ small.
Then, by the divergence theorem,
$$\int_{D_\eps} div \, V\;dX = \int_{\Gamma_{\phi_{\bar t}} \cap \p D_\eps} V \cdot \nu_{\phi_{\bar t}} \;d\sigma- \int_{\Gamma_u \cap \p D_\eps} V \cdot \nu_u \;d\sigma - \int_{\Gamma_\eps} V \cdot e_{n+1} \;dx,$$
and in view of \eqref{hom}-\eqref{con},
$$\int_{D_\eps} div \, V \;dX\geq \int_{\Gamma_{\phi_{\bar t}} \cap \p D_\eps} G(\phi_{\bar t}, \nu_{\phi_{\bar t}})\;d\sigma - \int_{\Gamma_u \cap \p D_\eps} G(u, \nu_u)\;d\sigma - \int_{\Gamma_\eps} V \cdot e_{n+1}\; dx.$$ From the formula for $V$, on $\Gamma_\eps$, for $\eps$ small,
$$V(x,\eps) \cdot e_{n+1}= - \frac 12 |\nabla \phi_{t_X}|^2 + \frac 1 \gamma  \phi_{t_X}^{-\gamma} \leq 0.$$
Indeed, we only need to verify that the one variable function of $d$, 
$$\phi(d):=c_0 d^\alpha + \frac{\mu}{2} d^\beta + d^\sigma,$$ 
satisfies:
$$\frac 12 \phi'^2 \geq \frac 1 \gamma \phi^{-\gamma}, \quad \text{if $d>0$ is small}.$$ 
Since $\mu>0$, we know that $\phi \geq c_0d^\alpha$ while $\phi' \geq \alpha c_0 d^{\alpha -1}.$ Hence, by the definition of $\alpha, c_0$ (see \eqref{a}),
$$ \phi'^2 \geq \alpha ^2 c_0^2 d^{2(\alpha -1)}= \alpha^2 c_0^2 d^{-\alpha \gamma} \geq \alpha^2 c_0^{2+\gamma} \phi^{-\gamma} = \frac 2 \gamma \phi^{-\gamma},$$ as desired.

Finally, this implies that, after letting $\eps \to 0$, 
\begin{equation}\label{divv}
\int_{D} div \, V  \;dX \geq \int_{\Gamma_{\phi_{\bar t}} \cap \p D} G(\phi_{\bar t}, \nu_{\phi_{\bar t}})\;d\sigma - \int_{\Gamma_u \cap \p D} G(u, \nu_u)\;d\sigma.
\end{equation} 
Next we show that 
$$div \, V = - \triangle \phi_{t_X} - \phi_{t_X}^{-(\gamma+1)} <0,$$
and the left hand side in the inequality \eqref{divv} is non-positive, which in view of the definition of $G$ contradicts the minimality  of $u$ (see \eqref{EE}). 

To compute $div \, V$ at a point $(z_0, \phi(z_0))$, let $$D_\varphi:= \{0<\phi(x) - \eps \varphi (x) < x_{n+1}< \phi(x), x \in B_\delta(z_0)\}$$ with $\varphi(z_0)> 0$ and $\varphi$ a smooth bump function supported on $B_{\delta}(z_0) \subset \{\varphi>0\}$.
Then, by a similar computation as above, 
$$\int_{D_\varphi} div \, V \;dX = \int_{\Gamma_{\phi}} G(\phi, \nu_{\phi})\;d\sigma -  \int_{\Gamma_{\phi-\eps \varphi}} G((\phi-\eta\varphi), \nu_{\phi-\eps\varphi})\;d\sigma + O(\eps^2),$$
where we used that  if $x_{n+1}=\phi(x)-\eps \varphi(x)= \phi_{t}(x),$ then $\nu_{\phi-\eps\varphi}(x)= \nu_{\phi_{t}}(x)+O(\eps)$ and by the homogeneity and $C^2$ smoothness of $G$
$$G(x_{n+1}, \nu_{\phi - \eps \varphi}(x)) = \nabla_y G (x_{n+1}, \nu_{\phi_{t}}(x))\cdot \nu_{\phi-\eps\varphi}(x)+ O(\eps^2).$$
Thus, for $\eps$ small,  
\begin{align*}
\int_{D_\varphi} div \, V \;dX=& \int_{B_\delta(z_0)}\left(\frac 12 |\nabla \phi|^2 + W(\phi)\right) dx\\
& \quad \quad  - \int_{B_\delta(z_0)}\left(\frac 12 |\nabla(\phi-\eps\varphi|^2 + W(\phi-\eps \varphi)\right) dx + O(\eps^2)\\
=&\eps \int_{B_\delta(z_0)} \left(\nabla \phi \cdot \nabla \varphi  - \phi^{-(\gamma+1)}\varphi\right) dx + O(\eps^2)\\
=& \eps \int_{B_\delta(z_0)}\left(-\Delta \phi - \phi^{-(\gamma+1)} \right)\varphi dx + O(\eps^2).
\end{align*}
We divide by $\eps$ and let $\eps \to 0$ and then $\delta \to 0$. Since $|D_{\varphi}|=\eps \int \varphi dx$, and $D_\varphi$ tends to $(z_0, \phi(z_0))$, 
we conclude that at $(z_0, \phi(z_0))$ $$div \ V= -\Delta \phi - \phi^{-(\gamma+1)}.$$ The desired conclusion follows by equation \eqref{sub}.
\end{proof}

\begin{rem}\label{comp} For the proof of the subsolution property of minimizers we need a slightly weaker condition than the one required in Definition \ref{defn}: a minimizer $u$ restricted to each connected component of $\{u>0\}$ which has $x_0$ on its boundary, cannot touch by below a comparison function $\psi \in \mathcal C^-$ (with $\mu<0$) at $x_0$.

\end{rem}

\begin{rem}\label{r1}In view of Proposition \ref{min=vis}, a minimizer satisfies the estimate \eqref{univ}. Notice that unlike Proposition \ref{opt1}, in this estimate the constant $C$ does not depend on the boundary data.\end{rem}

\begin{rem} \label{r2} A straightforward application of the maximum principle gives that a continuous nonnegative function $u$ that satisfies $\triangle u = - u^{-(\gamma+1)}$ in the set $\{u>0\}$ has the {\it weak non-degeneracy} property
$$ u(x) \ge c \, \, dist (x, F(u)) ^\alpha,$$
for some $c>0$. 

Indeed, if $B_1 \subset \{u>0\}$, then $u \geq c$ in $B_{1/2}$, with $c>0$ universal. For this it is enough to compare $u$ with the explicit radially symmetric solution of $\Delta \phi = - \phi^{-(\gamma +1)}$ which vanishes on $\p B_1$.
\end{rem}

We now prove a strong non-degeneracy property for minimizers of \eqref{J}, which combined with Lemma \ref{opt} implies Theorem \ref{T2}. 
\begin{prop}[Non-degeneracy] \label{SND}
Assume $u$ is a minimizer of \eqref{J} in $B_1$ and $0 \in F(u)$. Then
$$ \max_{\p B_r} u \ge c r^\alpha, \quad \quad r \le 1/2.$$
\end{prop}

We remark that strong non-degeneracy does not follow from the weak non-degeneracy of Remark \ref{r2} via a standard iterative argument as in the nonnegative power case \cite{AP, CS}. Instead, we prove the following lemma.

\begin{lem} Let $u$ minimize \eqref{J} in $B_2.$ There exists a universal constant $\delta>0$ such that 
$$\text{if $u \leq \delta$ on $\p B_1$ then $u\equiv 0$ in $B_{1/8}.$}$$
\end{lem}

\begin{proof} We prove the theorem with $B_{1+d_1}$ instead of $B_1$, with $d_1$ a small universal constant to be specified below. Let
$$\phi:= c_0 d^\alpha - d^{2-\alpha} - d^\sigma, \quad 2-\alpha< \sigma <4-3\alpha,$$ with $d(x):=dist(x, \p B_1)$
 when $|x|\geq 1$ and $0$ otherwise. Then, the computation from Lemma \ref{strict} shows that 
 $$\Delta \phi < - \phi^{-(\gamma+1)}, \quad \text{in $B_{1+d_0} \setminus \bar B_1.$}$$
 We choose $d_1(\delta)$ so that
 $$ \phi |_{\p B_{1+d_1}} = \delta.$$
 Then, since $u \leq \delta$ on $\p B_{1+d_1},$
 $$J(u, \{u> \phi\}) \leq J(\phi, \{u>\phi\}) \leq J(\phi, B_{1+d_1}) \to 0, \quad \text{as $\delta \to 0$.}$$
 In particular, for $o(1) \to 0$ as $\delta \to 0,$
 \begin{equation}\label{small1}J(u, B_{1})=o(1).\end{equation}
 Since $u \le \delta$ on $\p B_{1+d_1},$ the maximum principle implies that $u \le C$ in $B_{1+ d_1}$ by comparing $u$ with the radial solution which has boundary data $1$ on $\p B_{1+d_1}$. 
 
 
 We now consider the family of subsolutions $\phi_\lambda(x):= \lambda^{-\alpha} \phi(\lambda x)$ whose graphs foliate the region $D:=\{0 \leq x_{n+1} \leq \kappa |x|^\alpha\}$ with $\kappa = \phi |_{\p B_{1+d_0}}$. By the same calibration argument as in Proposition \ref{min=vis} we conclude that if $u$ is a minimizer defined in $B_{1+d_0}$, 
 $$\text{if $u \leq \kappa |x|^\alpha$ in $B_{1+d_0}$ then $u \leq \phi$ in $B_{1+d_0}$,}$$ and in particular $u \equiv 0$ in $B_1$. 
 Thus we have the following dichotomy: 
 $$\mbox{either $u \equiv 0$ in $B_1$ or there exists $\bar x \in B_{1+d_0}$ such that $u(\bar x) > \kappa |\bar x|^\alpha$.}$$
 The second alternative, by the optimal regularity Lemma \ref{opt}, implies that there exists $0<r=|\bar x| < 1+d_0$ such that $B_{c' r}(\bar x) \subset \{u>0\},$ for some $c'$ universal. 
 By a dilation of a bounded factor and a translation we conclude that for all $x_0 \in B_{1/8}$ 
$$\text{either $u \equiv 0$ in $B_{1/4}(x_0)$ or $|B_r(x_0) \cap \{u>0\}| \geq \bar c |B_r(x_0)|$ for some $r\le 1/2$.}$$ 
If for some $x_0$ we have that $u \equiv 0$ in $B_{1/4}(x_0)$ then $u \equiv 0$ in $B_{1/8}$ as desired. Assume by contradiction the second alternative holds at all $x_0$. Since $u$ is bounded above,
 $$J(u, B_{r}(x_0)) \geq c_1 |B_r(x_0)|,$$ for some $c_1>0$ universal. 
 We use a finite overlapping cover with these balls and find
 $$J(u, B_1) \geq c |B_{1/8}|,$$
 which contradicts \eqref{small1} for $\delta$ small enough.
\end{proof}
 

Next, we prove a simple lemma which will be used to obtain Weiss monotonicity formula in the following section.

\begin{lem}\label{u2}Let $u$ be a minimizer to $J$ in $B_1$, then $u^2 \in C^{0,1}(B_1).$
\end{lem}
\begin{proof} Let $x_0 \in \{u>0\} \cap B_{1/2}$ and $r:=dist(x_0,F(u))$ with $F(u) \cap B_{1/2} \neq \emptyset$. We rescale around $x_0$,
$$\tilde u(x):= \frac{u(x_0 + rx)}{r^\alpha},$$
and the optimal regularity and Remark \ref{r2} imply $\tilde u \sim 1$ in $B_{3/4}$. Thus by elliptic regularity, $\tilde u |\nabla \tilde u| \leq C$ in $B_{1/2}$. Rescaling back we find that $$u|\nabla u|(x_0) \leq r^{2\alpha -1},$$ and since $\alpha > 1/2$ the desired claim follows.\end{proof}

We conclude the section with the stability of non-degenerate viscosity solutions under uniform limits. The proof follows immediately from the definitions and we omit it.

\begin{prop} Let $u_k$ be a sequence of non-degenerate viscosity solutions  to \eqref{v} in $B_1$, and $u_k \to u$ uniformly locally in $B_1$. Then, $u$ is a viscosity solution to \eqref{v} in $B_1.$
\end{prop}


\section{Compactness of minimizers}

The main result of this section is the following compactness statement.

\begin{prop} Let $u_k$ be a sequence of minimizers to \eqref{J} in $B_1$ which converges uniformly to $u$ locally in $B_1$. Then, $u$ is a minimizer to \eqref{J} in $B_1.$
\end{prop}
\begin{proof} Let $v$ be an admissible competitor with $v=u$ in $B_1 \setminus \bar B_{1-\delta}$ for $\delta>0$ small. We use Lemma \ref{uv} below, and call $v_k$ the interpolation of $u_k$ and $v$ such that $$v_k= v \quad \text{in $B_{1-\delta,}$} \quad v_k=u_k \quad \text{in $B_1 \setminus \bar B_{1-\delta/2}$}.$$ Then, by the minimality of $u_k$ and Lemma \ref{uv}, we get
$$J(u_k, B_1) \leq J(v_k, B_1) \leq J(v, B_{1-\delta/2}) + J(u_k, B_1 \setminus \bar B_{1-\delta}) + o(1),$$ with $o(1) \to 0$ as $k\to \infty.$ Notice that, by Lemma \ref{EB}, $J(u_k,B_{1-\delta/2})$ are uniformly bounded, and the hypotheses of Lemma \ref{uv} apply.

After subtracting $J(u_k, B_1 \setminus \bar B_{1-\delta})$ from both sides we deduce that,
$$J(u_k, B_{1-\delta}) \leq  J(v,  B_{1-\delta/2}) + o(1).$$
By the lower semi-continuity of $J$, we obtain that,
$$J(u, B_{1-\delta}) \leq  J(v, B_{1-\delta/2}).$$ Our claim follows by letting $\delta \to 0.$
\end{proof}

In the next lemma we interpolate between two functions which are $L^2$ close in an annulus, without increasing too much the total energy.
\begin{lem}\label{uv} Let $u_k,v_k$ be sequences in $H^1(B_1)$ and $\delta>0$ small. Assume that $u_k-v_k \to 0$ in $L^2(B_{1-\delta/2} \setminus \bar B_{1-\delta}),$ as $k\to \infty,$ and that $u_k, v_k$ have uniformly (in $k$) bounded energy in $B_{1-\delta/2}$. Then, there exists $w_k \in H^1(B_1)$ with $$w_k:= \begin{cases}v_k \quad \text{in $B_{1-\delta}$}\\ u_k \quad \text{in $B_1 \setminus \bar B_{1-\delta/2}$}\end{cases}$$ such that 
 $$J(w_k, B_1) \leq J(u_k, B_{1-\delta/2}) + J(v_k, B_1 \setminus \bar B_{1-\delta}) + o(1),$$
 with $o(1) \to 0$ as $k \to \infty.$
\end{lem}
\begin{proof} For notational simplicity, in what follows we drop the subscript $k$.

Let  $\varphi(s)$ be a smooth function on $\R$, which is $0$ on $(-\infty, 0]$ and $1$ on $[1,+\infty)$, and such that  for $c>0$ small,
\begin{equation}\label{cut}
\varphi(s)= s^\alpha \quad \text{in [0, c]}, \quad \varphi(s)=1-(1-s)^\alpha, \quad \text{in $[1-c, 1].$}
\end{equation} Thus,
$$\int_0^1 (\varphi')^{2} ds,  \quad \int_0^1 (\varphi^{-\gamma}+(1-\varphi)^{-\gamma}) ds \leq C.$$
For $0< \mu< \frac{\delta}{4}$, we set
$$\varphi_\mu(s) = \varphi\left(\frac{s}{\mu}\right),$$
and obtain
\begin{equation}\label{phi1}\int_\R (\varphi'_\mu)^{2} ds \leq \frac{C}{\mu},\end{equation}  \begin{equation}\label{phi2}\int_\R (\varphi_\mu^{-\gamma}+(1-\varphi_\mu)^{-\gamma})\chi_{\{0<\varphi_\mu<1\}} ds \leq C\mu.\end{equation}
For all $r \in [1-\delta, 1-\frac 3 4 \delta]$, denote by $\varphi_r(x):=\varphi_\mu(|x|-r)$ and define
\begin{equation}\label{w_r}
w_r:= \varphi_r u + (1-\varphi_r)v.
\end{equation} 
Then,
\begin{align*}J(w_r, B_1) \leq &J(u, B_r) + J(v, B_1 \setminus \bar B_{r+\mu}) + \\
\ & C\int_{B_{r+\mu} \setminus \bar B_r} (|\nabla \varphi_r|^2(u-v)^2 + \varphi_r^2|\nabla u|^2+(1-\varphi_r)^2|\nabla v|^2)\,dx\\  + & \int_{B_{r+\mu} \setminus \bar B_r} ((\varphi_r u)^{-\gamma} \chi_{\{u>0\}}+ ((1-\varphi_r)v)^{-\gamma} \chi_{\{v>0\}}) \,dx.
\end{align*}
Thus,
\begin{align*}J(w_r, B_1) \leq &J(u, B_{1-\frac 3 4 \delta}) + J(v, B_1 \setminus \bar B_{1-\delta}) + \\
\ & \int_{B_{1-\delta/2} \setminus \bar B_{1-\delta}} G(r,x) \chi_{\{r< |x|<r+\mu\}}dx\end{align*}
where
\begin{align*}G(r,x):=& C(|\nabla \varphi_r|^2(u-v)^2 + \varphi_r^2|\nabla u|^2+(1-\varphi_r)^2|\nabla v|^2)+\\ &((\varphi_r u)^{-\gamma} \chi_{\{u>0\}}+ ((1-\varphi_r)v)^{-\gamma} \chi_{\{v>0\}})\end{align*}
We now average in $r$ and use Fubini on the right hand side, to obtain,
\begin{align*} \frac 4 \delta \int_{1-\delta}^{1-\frac 3  4 \delta} J(w_r, B_1) dr & \leq  (J(u, B_{1-\frac 3 4 \delta}) + J(v, B_1 \setminus \bar B_{1-\delta}))+\\
\ & \frac 4 \delta \int_{B_{1-\delta/2} \setminus \bar B_{1-\delta}} \int_{|x|-\mu}^{|x|} G(r,x) dr dx. 
\end{align*}
Using \eqref{phi1}-\eqref{phi2}, and recalling that $u=u_k$ and $v=v_k$, we therefore obtain:
\begin{align*} \frac 4 \delta \int_{1-\delta}^{1-\frac 3  4 \delta} J(w_r, B_1) dr  \leq  &(J(u_k, B_{1-\frac 3 4 \delta}) + J(v_k, B_1 \setminus \bar B_{1-\delta}))\\
& + \frac{C}{\mu\delta} \int_{B_{1-\delta/2} \setminus \bar B_{1-\delta}}(u_k-v_k)^2dx \\
& + \frac{C\mu}{\delta}(J(u_k, B_{1-\delta/2} \setminus \bar B_{1-\delta}) + J(v_k, B_{1-\delta/2} \setminus \bar B_{1-\delta})).
\end{align*}
By choosing first $\mu$ small and then $k$ large, we obtain the desired statement in view of the assumption that $\|u_k-v_k\|_{L^2} \to 0$ as $k \to \infty$.
\end{proof}

We conclude with a proposition which will be needed in the dimension reduction argument of the next section.

\begin{prop} \label{dim}Assume $u$ is constant in the $e_1$ direction, i.e.,  $$u(x_1,\ldots, x_n)=v(x_2, \ldots, x_n).$$ Then $u$ is a minimizer in $\R^n$ if and only if $v$ is a minimizer in $\R^{n-1}.$
\end{prop}
\begin{proof}
Assume $u$ is a minimizer in $\R^{n}$ and let $w (x_2,\ldots, x_{n})$ be a function which coincides with $v$ outside $B_K \subset \R^{n-1}.$ Then define $ \tilde u_r$ to be the interpolation between $w$ and $v$ defined in $\mathbb R \times B_K$ as
$$\tilde u_r:= \varphi(|x_1|-r)v(x_2,\ldots, x_{n}) + (1-\varphi(|x_1|-r))w(x_2,\ldots, x_{n}),$$
with $\varphi$ the function defined in \eqref{cut}. 

Notice that $\tilde u_r = w$ if $|x_1| <r$, and $u_r=v$ if $|x_1|>r+1$. From the minimality of $u$ we find
$$ J (v, \mathcal C_R) \le J(\tilde u_r, \mathcal C_R), \quad \mathcal C_R:=[-R,R] \times B_K,$$
provided that $r+1 \le R$. We integrate the inequality above in $r \in [R-2,R-1]$ and obtain
$$ J (v, \mathcal C_R) \le J(w, \mathcal C_{R-2}) + \int_{R-2}^{R-1} J(\tilde u_r, \mathcal C_R \setminus \mathcal C_{R-2}) dr.$$
As in the proof of Lemma \ref{uv}, the integral term above is bounded by a constant depending on $v$, $w$ and the universal constants but independent of $R$. We divide the inequality by $R$, and let $R \to \infty$ to obtain the desired inequality
$$ J (v, B_K) \le J(w, B_K).$$

Conversely, assume that $v$ is a minimizer in $\R^{n-1}.$ Let $w$ be a function of $n$ variables which coincides with $u$ outside of $\mathcal C_R$. Then 
$$J(w, \mathcal C_R) \geq \int_{-R}^R J(w(x_1,\cdot), B_K) dx_1,$$ and by the minimality of $v$,
 $$J(w,  \mathcal C_R) \geq \int_{-R}^R J(v(x_2,\ldots, x_{n}), B_K) dx_1= J(u,\mathcal C_R).$$\end{proof}


\section{Weiss monotonicity formula and consequences}

In this section we prove Weiss monotonicity formula (\cite{W,PSU}) for minimizers
of the energy functional $J$, and derive the partial regularity result Theorem \ref{T3}.

\begin{thm} If $u$ is a minimizer to $J$ in $B_R$ then
$$W_u(r):= r^{-n-2(\alpha-1)} J(u, B_r) - \alpha r^{-(n-1)-2\alpha} \int_{\p B_r} u^2 d\sigma, \quad 0<r\leq R,$$ is increasing in $r$. Moreover, $W_u$ is constant if and only if $u$ is homogeneous of degree $\alpha$.\end{thm}

Notice that the optimal regularity Lemma \ref{opt} implies that if $0 \in F(u)$ then $W_u(r)$ is bounded below as $r \to 0$.
\begin{proof} In view of Lemma \ref{u2}, $W_u(r)$ is differentiable for a.e. $r$ and by standard computations
$$\frac{d}{dr} J(u,B_r) = \int_{\p B_r} (\frac 1 2 |\nabla u|^2 + W(u)) dx,$$
while
$$\frac{d}{dr} \left(r^{-(n-1)-2\alpha}\int_{\p B_r} u^2\right) d\sigma = 2r^{-n-2\alpha}\int_{\p B_r}  (r uu_\nu - \alpha u^2) d\sigma.$$
Assume that these equalities are satisfied at $r=1$.
Then,
\begin{align*}\frac{d W_u}{dr} |_{r=1}= \int_{\p B_1}  (\frac 1 2 |\nabla u|^2 + W(u)) d\sigma - & (n+1+2\alpha)J(u,B_1)\\ \ & - 2\alpha\int_{\p B_1}  ( uu_\nu - \alpha u^2) d\sigma,\end{align*}
from which we deduce
\begin{align*}\frac{d W_u}{dr} |_{r=1}= \int_{\p B_1}  (\frac 1 2 u_\tau^2 + W(u)) d\sigma + & \frac{\alpha^2}{2} \int_{\p B_1} u^2d\sigma -  (n+1+2\alpha)J(u,B_1)\\ \ & + \frac 1 2\int_{\p B_1}  (u_\nu - \alpha u)^2 d\sigma.\end{align*}
We claim that 
$$I(u):=\int_{\p B_1}  (\frac 1 2 u_\tau^2 + W(u)) d\sigma +  \frac{\alpha^2}{2} \int_{\p B_1} u^2d\sigma \geq  (n+1+2\alpha)J(u,B_1).$$
Indeed let $$\tilde u(x):= |x|^\alpha u\left(\frac{x}{|x|}\right), \quad x \in B_1$$ the $\alpha$-homogeneous extension of $u|_{\p B_1}$.  Then
$$I(u)= I(\tilde u)= (n+1+2\alpha)J(\tilde u, B_1),$$
where the last equality follows from the computation above for $\frac{d}{dr}W_{\tilde u}$ which is $0$. and our claim follows from minimality. Thus,
$$\frac{d}{dr} W_u(r) \geq 0 , \quad \text{a.e. $r$}.$$ The conclusion follows since $W_u(r)$ is absolutely continuous in $r$. Moreover, the computations above show that $W_u$ is constant if and only if $$u_\nu = \frac{\alpha}{|x|} u, \quad \text{for a.e. $x$,}$$ that is $u$ is homogeneous of degree $\alpha$.
\end{proof}

Next we study the homogeneous global minimizers in the 2D case. 

\begin{prop}
The homogenous of degree $\alpha$ minimizers of $J$ in dimension $n=2$ are rotations of $c_0 (x_1^+)^\alpha$.
\end{prop}

\begin{proof} Let $u$ be a global minimizer to \eqref{J} in $\mathbb{R}^2$, homogeneous of degree $\alpha$. Then $u(r,\theta)= r^\alpha f(\theta)$ with $f$ a H\"older continuous function solves the following ODE in each interval of $\{f>0\}$
\begin{equation}\label{ODE}f'' + \alpha^2 f= - f^{-(\gamma+1)}.
\end{equation}
Notice that $f$ cannot be positive everywhere as we would contradict the ODE at the minimum point. 
Assume that $(0,a)$ is a maximal interval of 
$\{f>0\}$. 

{\it Claim:} $u=c_0 x_2^ \alpha$ in the half-space $\{x_2 \ge 0\}$.

We multiply the ODE by $f'$ and integrate, and deduce that in $(0,a)$
$$\frac 1 2 (f'^2+\alpha^2 f^2) - W(f)= \mu,$$
for some constant $\mu \in \R.$ We wish to show that $\mu=0$, which gives that $f$ is uniquely determined up to translations, and then the claim follows.

We argue as in Section \ref{the1d}, and deduce that $f$ has the same two first terms in the expansion near $0$ (see \eqref{oned}),
$$f(\theta)=c_0\theta^\alpha+ \mu c_1 \theta^{2-\alpha} + O(\theta^{\sigma}), \quad \quad \sigma>2-\alpha,$$
with $c_1>0$ a universal constant.
In view of Theorem \ref{min=vis} and Remark \ref{comp}, $u$ satisfies the free boundary condition at $e_1 \in F(u)$. This gives $\mu=0$ and the claim is proved.

It remains to rule out the case when $u=c_0|x_2|^\alpha$. This follows from Proposition \ref{dim}, and the fact that $c_0 |t|^\alpha$ is not a minimizer of $J$ in $\mathbb R$.

\end{proof}

In the next section we show that if a viscosity solution is close to the one-dimensional solution then its free boundary is $C^{1,\beta}$ (see Proposition \ref{P1}). As a consequence we have the following result.

\begin{thm}\label{T5}
Let $u$ be a minimizer of $J$ in $B_1$ such that $$\|u- c_0(x_n^+)^\alpha \|_{L^\infty} \le \eps_0(\gamma,n).$$ Then $F(u) \cap B_{1/2}$ is a $C^{1,\beta}$ graph in the $x_n$ direction. 
\end{thm}

 Indeed, from the non-degeneracy property of minimizers we conclude that $u$ is trapped between two translations of $(x_n^+)^\alpha$ and then Proposition \ref{P1} of the next section applied to $w:= (u/c_0)^{1/\alpha}$ gives the interior $C^{1,\beta}$ estimate for $F(u)$.
 
A standard consequence of the Weiss monotonicity formula together with the compactness of minimizers is the following energy gap for cones. 
  
  \begin{prop}[Energy gap]
  Assume that $U$ is a cone, i.e. nonzero homogenous of degree $\alpha$ minimizer of $J$, and let $U_0:= (x_n^+)^\alpha$ denote the one-dimensional solution. If $U$ is not one-dimensional then
  $$W(U) \ge W(U_0) + \delta,$$
  for some $\delta>0$ universal depending only on $n$ and $\gamma$. 
  \end{prop}
 
  At this point we have all the ingredients to perform the dimension reduction argument of Federer \cite{F}. The next results follow from the standard techniques in free boundaries and we omit the proofs (see \cite{DS1} for more details).

 \begin{thm}
 Let $u$ be a minimizer for $J$ in $B_1$. Then 
 $$ \mathcal H^{n-1}(F(u) \cap B_{1/2}) \le C(n, \gamma) $$
 and $F(u)$ is locally $C^{1,\beta}$ except on a closed singular set $\Sigma_u \subset F(u)$ of Hausdorff dimension $n-3$. 
 \end{thm}
 
 
\section{Improvement of flatness}

In this section we show that viscosity solutions to our problem that are sufficiently close to a 1D solution have $C^{1,\beta}$ free boundaries.
We follow the same strategy as in our previous work \cite{DS2} which includes the case of positive powers. 

After a change of variables we rewrite the equation in the form
$$\triangle w = \frac{h(\nabla w)}{w}.$$
Precisely, we denote 
$$ w:= c_0^ {- \frac 1 \alpha} u^\frac 1 \alpha,$$
so that $u = c_0 w^\alpha$. The equation for $w$ is
$$c_0 \alpha w ^{\alpha -2} \left( w\triangle w + (\alpha -1)  |\nabla w|^2 \right) = - c_0^{-(\gamma+1)} w^{-\alpha(\gamma+1)},$$
and using \eqref{a}
\begin{equation}\label{weq}
 \triangle w = (1-\alpha) \frac{|\nabla w|^2 -1}{w} = :\frac{h(\nabla w)}{w}.
 \end{equation}
Here $h$ is a radial quadratic function which vanishes on $\p B_1$, it is negative in $B_1$ and positive outside $B_1$ and 
\begin{equation}\label{nablah}
 \nabla h( \omega) = - s \, \omega, \quad \mbox{if} \quad \omega \in \p B_1, \quad s :=2(\alpha-1) \in (-1,0). 
 \end{equation}
Notice that \eqref{weq} remains invariant under the rescaling
$$\ \tilde w(x)=\frac {w(rx)} {r}.$$
In view of the viscosity definition for $u$, we find that $w$ satisfies \eqref{weq} with the following free boundary condition on $\p \{w>0\}$:

\begin{defn}\label{defnw} We say that $w: \Omega \to \mathbb R^+$ satisfies \eqref{weq} in the viscosity sense, if $w$ is $C^\infty$ and satisfies the equation in the set $\{w>0\} \cap \Omega$ and, if $x_0 \in F(w):= \p \{w>0\} \cap \Omega$, then $w$ cannot touch $\psi \in \mathcal C^+$ (resp. $\mathcal C^-$) by above (resp. below) at $x_0$, with $$\psi(x):= d(x) + \mu \,  d(x)^{3 - 2 \alpha},$$ $\alpha$ as in \eqref{a} and $\mu>0$ (resp $\mu<0$).
\end{defn}
Notice that 
$$ 3 - 2 \alpha = 1-s \, >1. $$
\begin{prop}\label{P1}
Assume that $w$ is a viscosity solution of \eqref{weq} in $B_1$, and
$$ (x_n- \eps)^+ \le w \le (x_n+ \eps)^+, \quad \quad 0 \in F(w),$$
for some $\eps \le \eps_0(\gamma,n)$ small, universal. Then, there exists $r$ universal such that
$$ (x \cdot \nu - \frac \eps 2 r)^+ \le w \le (x \cdot \nu + \frac \eps 2 r)^+ \quad \mbox{in $B_r$},$$
for some unit direction $\nu$, $|\nu|=1$ and $|\nu-e_n| \le C \eps$.
\end{prop}

The strategy is to show that the rescaled function
\begin{equation}\label{tidw}
 \tilde w = \frac{w-x_n}{\eps} 
 \end{equation}
is well approximated by a viscosity solution of the linearized equation
\begin{equation}\label{LiE}
\begin{cases}
\Delta \varphi + s \dfrac{ \varphi_n}{x_n}=0, \quad \text{in $B_{1}^+,$}\\
\frac{\partial }{\partial x_n^{1-s}} \, \varphi =0 \quad \text{on $\{x_n=0\}$}.
\end{cases}
\end{equation}
We recall the definition from \cite{DS2} that $\varphi: \overline {B}_1^+ \to \mathbb R$ is a viscosity solution of the equation above if 

a) $\varphi$ is continuous up to the boundary 

b) $\varphi$ satisfies the equation above in $B_1^+$ in the classical sense 

c) $\varphi$ cannot be touched by below (above) at a point on $\{x_n =0\} \cap B_1$ by a test function
$$q(x):= a |x'- y'|^2 + b + p x_n^{1-s}, \quad \quad a,b \in \mathbb R, y' \in \mathbb {R}^{n-1},$$
with $p>0$ ($p<0$.) 

\

In \cite{DS2} we proved the following $C^{1,\sigma}$ estimate for solutions of \eqref{LiE}, see Theorem 7.2 in \cite{DS2}.

\begin{thm}\label{T7.2}
Assume that $\varphi$ is a solution of \eqref{LiE}, and $s > -1$. Then
$$|\varphi(x)- \varphi(0) - a' \cdot x' | \le C \|\varphi\|_{L^\infty} \,  |x|^{1+\sigma},$$
with $C$ large, $\sigma >0$ small, depending only on $n$ and $s$.
\end{thm}

In order to prove the convergence of the rescalings $\tilde w$ to a solution of \eqref{LiE} we first establish a Harnack type inequality for $\tilde w$ in the set $\{w>0\}$. 

\begin{lem}\label{HI}
Assume that $w$ is a viscosity solution in $B_1$ and 
$$ x_n +a - \eps \le w \le (x_n + a + \eps)^+, \quad \mbox{for some $a$, $|a| \le 1/10$.} $$If 
\begin{equation}\label{wgexn}
w \ge x_n+a  \quad \mbox{at} \quad x_0= \frac 12 e_n,
\end{equation} 
then
$$w \ge x_n + a+ (c-1) \eps \quad \mbox{in $B_{1/2}$.}$$
Similarly, if $$ w \le x_n+a  \quad \mbox{at} \quad x_0= \frac 12 e_n,$$ then
$$w \le (x_n + a+ (1-c) \eps)^+ \quad \mbox{in $B_{1/2}$.}$$
\end{lem}

\begin{proof}

After a translation, we may assume for simplicity that $a=\eps$ and $w \ge x_n$ in $B_1$. 

Since $w$ is trapped between two flat solutions $x_n$ and $x_n + 2 \eps$ which solve the same equation \eqref{weq}, we find that the interior estimates and the Harnack inequality \eqref{w_int} below continue to hold for the difference between $w$ and $x_n$ (see \cite{S}). 
Alternatively, if $x_n \ge \frac \sigma 2$, for some $\sigma>0$ small depending on $n$ and $s$ to be specified later, then $w$ is bounded below by $c(\sigma)$ thus (see \eqref{tidw})
$$|\triangle \tilde w| =\frac {1}{ \eps w} |h(e_n+ \eps \nabla \tilde w )| \le C(\sigma) |\nabla \tilde w| \quad \mbox{in} \quad \{ x_n \ge \frac \sigma 2\} \cap \{|\nabla \tilde w| \le \frac 1 \eps\}.$$
Then, according to Lemma 3.8 in \cite{DS2} we find $|\nabla \tilde w| \le C (\sigma)$ in $\{x_n \ge \sigma\} \cap B_{3/4}$. 
In particular if we are in the situation \eqref{wgexn}, then $\tilde w(x_0) \ge 1$, and we can apply Harnack inequality for $\tilde w \ge 0$
\begin{equation}\label{w_int}
w \ge x_n  + c(\sigma) \eps  \quad \mbox{in} \quad \{|x'| \le \frac 12\} \times \{ \sigma \le x_n \le \frac 12  \}.
\end{equation}
  This proves the conclusion in $B_{1/2} \cap \{x_n \ge \sigma\}$, and it remains to prove that a similar inequality can be extended near $x_n=0$. Towards this aim we compare $w$ with an explicit subsolution $\Psi$ that we construct below.

Let $\mu>0$ be a small universal constant, and denote by $B:=B_{\frac {1}{\mu \eps}}(\frac {1}{\mu \eps} e_n)$ the ball of radius $\frac {1}{\mu \eps}$ centered at $\frac {1}{\mu \eps} e_n$. Also, let $d(x)$ denote the distance from $x$ to the boundary of the ball $B$ when $x \in B$, and extend $d(x)=0$ outside $B$.
In the cylinder $$\mathcal C :=\{|x'| \le \frac 12\} \times \{ |x_n| \le \sigma  \},$$ we compare $w$ and translations of 
$$\Psi=d + \mu \eps (d^{1-s} + A d^{\beta}),$$
for some $\beta$ in the interval $1-s < \beta \le \min\{ 1- 2s, 2\}$. The constant $A$ is large universal, chosen such that $\Psi$ is a subsolution to the equation \eqref{weq}. Indeed,
\begin{align*}
 \triangle \Psi  = \mu \eps  &   \left((1-s)(-s) d^{-s-1} + \beta (\beta-1)A d^{\beta -2}\right) \\
 & - \frac{n-1}{\frac{1}{\eps \mu} - d} \left (1+ \mu \eps (1-s) d^{-s} + \mu \eps A \beta d^{\beta-1}\right) \\
 \ge \eps \mu & \left((1-s)(-s)\cdot d^{-s-1} + \beta (\beta-1) A d^{\beta-2} +  \frac{2(n-1)}{s}\right) ,
\end{align*}
while, by \eqref{nablah}, $$|\nabla h(z)| = (-s) (|z|-1) + O((|z|-1)^2),$$
and $\Psi \ge d$ implies
$$ \frac{h(\nabla \Psi)}{\Psi} \le (-s) \mu \eps \frac{ (1-s)d^{-s} +  A\beta d^{\beta-1}}{d} + O ((\mu \eps)^2 d^{-2s-1}).$$
Thus, if $d \in (0,2)$ then
\begin{align*}
 \triangle \Psi - \frac{h(\nabla \Psi)}{\Psi}  & \ge \mu  \eps  \left (\beta (\beta -1+s) A d^{\beta-2}  + \frac{2(n-1)}{s}\right) \\
 & \quad + O ((\mu \eps)^2 d^{-2s-1}) \\
  & >0,  
 \end{align*}
 provided that $A$ is chosen large depending on $n$ and $s$, and $\eps$ is sufficiently small. Notice also that by construction, $\Psi$ satisfies the viscosity subsolution property of Definition \ref{defnw} on the free boundary $\p B$.

Next we check that on the boundary of $\mathcal C$ where either $|x'|=\frac 12$ or $x_n= \sigma$ we have
\begin{equation}\label{500}
w \ge \Psi (x + \frac{\mu \eps}{32} e_n).
\end{equation}
Indeed, 
$$d \le (x_n - \frac {\mu \eps} {4} |x'|^2)^+,$$ 
and in the set $\Psi >0$ where $|x'|=\frac 12$, $|x_n| \le \sigma$ we have
$$ \Psi \le x_n - \frac{\mu \eps}{4} |x'|^2 + \mu \eps (x_n ^{1-s} + A x_n^\beta)  \le x_n - \frac{\mu \eps}{32},$$
provided that $\sigma$ is chosen sufficiently small. This implies that
$$ \Psi (x + \frac{\mu \eps}{32} e_n) \le  x_n^+  \le w,$$
on the boundary of $\mathcal C$ where $|x'|=1/2$.

When $0 \le x_n \le 2 \sigma$, we use $d \le x_n$ and obtain
$$\Psi \le x_n + \eps \mu (x_n +  A x_n^\beta) \le x_n + c(\sigma) \eps - \frac{\mu \eps}{32},$$
provided that $\mu$ is chosen small, depending on $c(\sigma)$. Thus, on the part of the boundary of $\mathcal C$ where $x_n=\sigma$ we use \eqref{w_int} and obtain $ \Psi (x + \frac{\mu \eps}{32} e_n)\le w$, which proves the claim \eqref{500}. By comparing $w$ with a continuous family of translations $ \Psi (x + t e_n)$, with $t \le \frac{\mu \eps}{32} $ we conclude that the inequality \eqref{500} holds also in the interior of $\mathcal C$. This implies the desired inequality in a neighborhood of the origin which gives the conclusion of the lemma. 
\end{proof}

\begin{proof}[Proof of Proposition $\ref{P1}$]

Assume that for a sequence of $\eps_k \to 0$ and $w_k$ the conclusion does not hold for some $r$ universal, to be made precise later. By applying Lemma \ref{HI} repeatedly we conclude that, after passing to a subsequence, the graphs of
$$ \tilde w_k := \frac{w_k-x_n}{\eps_k} $$
defined in $ \overline{\{w_k >0\}}$, converge uniformly on compact sets to the graph of a H\"older limiting function $\bar w$  define in $\overline {B}_{1}^+$, and $$|\bar w|\le 1 \quad \bar w(0)=0,$$ 
since $|\tilde w_k| \le 1$, $\tilde w_k(0)=0$. The function $\tilde w_k$ solves the equation
$$ \triangle w_k = \frac {1}{ \eps_k} \frac{h(e_n+ \eps_k \nabla \tilde w_k)}{x_n + \eps_k \tilde w_k }:= g(\eps_k,x_n, \tilde w_k,\nabla \tilde w_k)$$
and 
$$g(\eps_k,x_n, z, p) \to \frac {\nabla h(e_n)  \cdot p}{x_n}=-s \frac{p_n}{x_n} \quad \mbox{as $k \to \infty$,}$$
we find that $\bar w$ solves the linear equation
\begin{equation}\label{501}
\triangle \bar w + s \frac{\bar w_n}{x_n} =0 \quad \mbox{in} \quad B_1^+,
\end{equation}
in the viscosity sense.

{\it Claim:} $\bar w$ satisfies the boundary condition of \eqref{LiE} on $\{x_n=0\}$. 

Assume by contradiction that, say for simplicity $\bar w$ is touched by below at $0$ by 
$$ - a |x'|^2 + p x_n^{1-s} $$
for some constants $a$, $p>0$. After, relabeling $p/2$ by $p$, and $a/2$ by $a$, we may assume that $w$ is touched strictly by below at $0$ by the function
$$q(x):= - a ( |x'|^2 - A x_n^ \beta) + p x_n^{1-s} , \quad \quad p>0,$$
with $\beta$ in the interval $1-s < \beta < \min \{ 1-2s, 2\}$, and $A$ is large such that $q$ is a subsolution of the equation \eqref{501} (notice that $x_n^{1-s}$ is a solution of \eqref{501}).

We construct the function $\Psi$ 
$$ \Psi= d + \eps p d^{1-s} + \eps a A d^\beta$$
where, as in Lemma \ref{HI}, $d=d(x)$ is the distance from $x$ to the boundary of the ball $B=B_{\frac {1}{2 \eps a}}(\frac {1}{2 \eps a} e_n)$ of radius $\frac {1}{2 \eps a}$ and center $\frac {1}{2 \eps a} e_n$. The same computation as in Lemma \ref{HI} shows that $\Psi$ is a subsolution to the equation \eqref{weq}. Moreover, using that $$d=(x_n - \eps a |x'|^2)^+ + O(\eps^2) $$
we find that
$$\tilde \Psi:=\frac{\Psi - x_n}{\eps} \quad \mbox{in} \quad \overline{\{ \Psi>0\}},$$
converges uniformly to $q(x)$ as $\eps \to 0$. Since $q(x)$ touches strictly by below $\bar w$ at the origin, and $w_k$ converge uniformly to $\bar w$, we conclude that a small $e_n$-translation of the graph $\Psi$ restricted to $\overline{\{ \Psi>0\}}$ (with $\eps=\eps_k$) must touch by below the graph of $w_k$ at a point $x_k \to 0$. We reached a contradiction since $\Phi$ is a strict subsolution, and the claim is proved.

\

Next we apply Theorem \ref{T7.2} to $\bar w$ and conclude that
$$|\bar w-a' \cdot x'| \le \frac{r}{8} \quad \mbox{in} \quad \overline {B}_r^+,$$
for some $r>0$ universal depending only on $n$ and $s$. This implies
$$\left(x_n + \eps_k (a' \cdot x' - \frac r 4)\right)^+ \le w_k \le (x_n + \eps_k (a' \cdot x' + \frac r 4))^+  \quad \mbox{in} \quad \overline B_r^+,$$
holds for large $k$. Then the conclusion is satisfied for $w_k$ with $\nu_k = \frac{e_n + \eps_k a'}{|e_n + \eps_k a'|}$ and we reached a contradiction.
\end{proof}

\end{document}